\documentstyle[jalc]{article}

\def\phi{\varphi}
\def\0{\varnothing}

\def\N{\bf N}

\def\ol{\overline}

\newtheorem{The}{Theorem}
\newtheorem{Lem}[The]{Lemma}

\theorembodyfont{\rm}

\title{ON THE PALINDROMIC DECOMPOSITION OF BINARY WORDS}
\author{Olexandr Ravsky}
\address{Department of Mathematics, Ivan Franko National Lviv University\\
      Universytetska 1, Lviv, Ukraine}
\email{oravsky@mail.ru}

\abstract{ We prove a precise formula for the minimal number $K(n)$ such that
every binary word of length $n$ can be divided into $K(n)$ palindromes.
Also we estimate the average number $\ol K(n)$ of palindromes composing a
random binary word of the length $n.$} 

\begin{document}
\maketitle
\section{Introduction}

The present note arose from the following problem proposed at International
Mathematical Tournament of Towns \cite{4}, p.8: {\it Prove that every binary
word of length 60 can be divided into 24 symmetric subwords and that the number
24 cannot be replaced by 14.} A word $w=a_0\dots a_n$ is called {\it symmetric}
if $a_i=a_{n-i}$ for all $i\le n$. For symmetric words we shall use a more
poetic term {\it palindrome}. Let $S$ be the set of nonempty binary words over
the alphabet $\{a,b\}$ and $S^1$ be the set $S$ with added the empty word.
Observe that the set $S^1$ is a semigroup with respect to the operation of
concatenation. The length of a word $w\in S^1$ will be denoted by $l(w)$. In
particular, the empty word has length $0$.

The above tournament task suggests three general problems:{\it

(1) Given a word $w\in S$ find the minimal number $m(w)$ of palindromes whose
product in $S$ is equal to $w$ (thus the number $m(w)$ can be thought as a
measure of asymmetry of $w$);

(2) Given a positive integer $n$ find the number $K(n)=\max\{m(w):l(w)=n\}$
equal to the maximal asymmetry measure of the ``worst'' binary word of length
$n$;

(3) Estimate the average asymmetry measure $\ol K(n)=2^{-n}\sum\{m(w):l(w)=n\}$
of a random binary word of length $n$.}

It should be noted that the first two questions were considered in \cite{5} and
\cite{2} while the last question was suggested to the author by O.Verbitsky.
Observe that the above problems are consistent only for a two-letter alphabet:
for every positive integer $n$ the word $(abc)^n$ in the three-letter alphabet
$\{a,b,c\}$ contains only trivial symmetric subwords.

For small numbers $n$ it turned to be possible to calculate the numbers $K(n)$
by computer:

\begin{table}[ht]
\begin{center}
\begin{tabular}{|c|c|c|c|c|c|c|c|c|c|c|c|c|c|c|c|}
\hline
$n$ & 1 & 2 & 3 & 4 & 5 & 6 & 7 & 8 & 9 & 10 & 11 & 12 & 13 & 14 & 15\\
\hline
$K(n)$ & 1 & 2 & 2 & 2 & 2 & 3 & 3 & 4 & 4 & 4 & 5 & 5 & 5 & 6 & 6\\
\hline
$n$ & 16 & 17 & 18 & 19 & 20 & 21 & 22 & 23 & 24 & 25 & 26 & 27 & 28 & 29 & 30\\
\hline
$K(n)$ & 6 & 6 & 7 & 7 & 8 & 8 & 8 & 8 & 9 & 9 & 10 & 10 & 10 & 10 & 11\\
\hline
\end{tabular}
\end{center}
\end{table}

This data allowed us to suggest and prove a precise formula for $K(n)$:

\begin{The} $K(n)=[\frac n6]+[\frac{n+4}6]+1$ for every
number $n\not=11$ and $K(11)=5$.
\end{The}

The number $n=11$ is exceptional and the word of length 11 destroying the
uniformity is  $w=aababbaabab$. The computer calculation shows that $w$ is a
unique word of length $11$ (up to change $a\leftrightarrow b$ and reading the
word from the right) with $m(w)=5$.

Theorem 1 will be proved by induction whose base uses the computer calculation
of $K(n)$'s for  $n\le 29$. Let us remark that the same values of $K(n)$ for
$n\le 29$ were independently obtained by Aleksandr Spivak \cite{2} which also
suggested a similar formula for $K(n)$.

Theorem 1 shows that the ``worst'' word of length $n$ is very asymmetric: it
cannot be divided into $<n/3$ palindromes. Next, we show that a random binary
word also is far from being symmetric: it cannot be divided into $<n/11$
palindromes. Like in the case of asymmetry measure $K(n)$ of a ``worst'' word
of length $n$, we start with computer calculation of the asymmetry measure $\ol
K(n)$ of a random word of length $n$ for small numbers $n$.

\begin{table}[hl]
\begin{center}
\begin{tabular}{|c|c|c|c|c|c|c|c|c|}
\hline
$ n $ & $\ol K(n)$ & $\ol K(n)/n$ &
$ n $ & $\ol K(n)$ & $\ol K(n)/n$ &
$ n $ & $\ol K(n)$ & $\ol K(n)/n$\\
\hline
$ 1 $ & $1.00 $ & $ 1.0000 $ &
$ 8 $ & $2.33 $ & $ 0.2910 $ &
$ 15$ & $3.36 $ & $ 0.2239 $\\
\hline
$ 2 $ & $1.50 $ & $ 0.7500 $ &
$ 9 $ & $2.46 $ & $ 0.2734 $ &
$ 16$ & $3.50 $ & $ 0.2188 $\\
\hline
$ 3 $ & $1.50 $ & $ 0.5000 $ &
$ 10$ & $2.61 $ & $ 0.2613 $ &
$ 17$ & $3.66 $ & $ 0.2152 $\\
\hline
$ 4 $ & $1.75 $ & $ 0.4375 $ &
$ 11$ & $2.75 $ & $ 0.2502 $ &
$ 18$ & $3.81 $ & $ 0.2114 $\\
\hline
$ 5 $ & $1.75 $ & $ 0.3500 $ &
$ 12$ & $2.91 $ & $ 0.2425 $ &
$ 19$ & $3.96 $ & $ 0.2084 $\\
\hline
$ 6 $ & $2.06 $ & $ 0.3438 $ &
$ 13$ & $3.05 $ & $ 0.2349 $ &
$ 20$ & $4.11 $ & $ 0.2055 $\\
\hline
$ 7 $ & $2.09 $ & $ 0.2991 $ &
$ 14$ & $3.20 $ & $ 0.2285 $ &
$ 21$ & $4.26 $ & $ 0.2030 $\\
\hline
\end{tabular}
\end{center}
\end{table}

This table will help us to estimate the limit $\ol K=
\lim\limits_{n\to\infty}\frac{\ol K(n)}n.$

\begin{The} The limit
$\ol K=\lim\limits_{n\to\infty}\frac{\ol K(n)}n$ exists,
is equal to $\inf\limits_{n\in\N}\frac{\ol K(n)}n$ and can be estimated
as $0.08781\dots<\ol K\le 0.2030\dots.$
\end{The}

To get the upper bound for $\ol K$ we use the results of computer calculation
while the lower bound is proved by a subtle analytic argument. From the table
we can expect that the exact value of $\ol K$ is close to $1/5$. It suggests
that an average binary word $w$ can be divided into $5/l(w)$ palindromes with
average length $5$.

\section{Proof of Theorem 1}
The proof of Theorem 1 is divided into eight lemmas. We start from the upper
bound. Let $w_n$ denote the word $aabab(bbaaba)^n$ and put $m_0=2$, $m_1=3$,
$m_2=3$, $m_3=3$, $m_4=4$, $m_5=4.$

\begin{Lem} For every
$n\ge 0$ we have

$m(w_n)\le 2n+m_0.$

$m(w_nb)\le 2n+m_1.$

$m(w_nbb)\le 2n+m_2.$

$m(w_nbba)\le 2n+m_3.$

$m(w_nbbaa)\le 2n+m_4.$

$m(w_nbbaab)\le 2n+m_5.$
\end{Lem}
\begin{proof}
For $n=0$ the lemma results from the following decompositions:

$w_0=(aa)(bab),$

$w_0b=(a)(aba)(bb),$

$w_0bb=(a)(aba)(bbb),$

$w_0bba=(aa)(b)(abbba),$

$w_0bbaa=(aa)(b)(abbba)(a),$

$w_0bbaab=(a)(aba)(bb)(baab).$

Suppose that we have already proved the lemma for $n=k.$ Then

$m(w_{k+1})=m(w_kbba(aba))\le 2k+3+1=2(k+1)+2$

$m(w_{k+1}b)=m(w_kbbaa(bab))\le 2k+4+1=2(k+1)+3$

$m(w_{k+1}bb)=m(w_kbbaaba(bb))\le 2(k+1)+2+1=2(k+1)+3$

$m(w_{k+1}bba)=m(w_kbbaab(abba))\le 2k+4+1=2(k+1)+3$

$m(w_{k+1}bbaa)=m(w_kbbaababb(aa))\le 2(k+1)+3+1=2(k+1)+4$

$m(w_{k+1}bbaab)=m(w_kbbaabab(baab))\le 2(k+1)+3+1=2(k+1)+4.$
\end{proof}
The following two lemmas are proved by routine computer calculations.

\begin{Lem} Let $u\in S, l(u)=6, w\in\{(bbaaba)^2bu,
(bbaaba)bbaaabau,bbaaabababbaau\}.$ Then one of the following conditions is
satisfied:

$1.$ $(\exists v',w'\in S):w\in v'w'S^1,l(v')<6$ and $m_{l(v')}+m(w')<
(5+l(v')+l(w'))/3.$

$2.$ $(\exists v',w'\in S):w\in v'w'S^1,l(v')<6,$ $m_{l(v')}+m(w')\le
(5+l(v')+l(w'))/3$ and $l(v')+l(w')+5$ is a multiple of $6.$

$3.$ $w\in\{(bbaaba)^3b,(bbaaba)^2bbaaaba,(bbaaba)bbaaabababbaa\}.$
\end{Lem}

\begin{Lem} Let $u\in S$, $12\le l(u)<18$ and $w\in\{(bbaaba)^2bu,
(bbaaba)bbaaabau,$\linebreak $bbaaabababbaau \}.$
Then one of the following conditions is satisfied:

\item{1.} There exist words $v',w'\in S$ such that $w=v'w',l(v')<5$ and
$m_{l(v')}+m(w')\le [17/2+l(u)/4].$
\item{2.} $w\in\{(bbaaba)^2bbaaababbbaaababba,(bbaaba)^3bbaaabababba\}.$
\end{Lem}

\begin{Lem} Let $w\in aS$, $l(w)=6n$, $n\ge 3.$ Then one of
the following conditions is satisfied:

1. $(\exists w'\in S):w\in w'S^1$ and $m(w')<l(w')/3.$

2. $(\exists w'\in S):w\in w'S^1$ and $m(w')\le l(w')/3$ and $l(w')$ is a
multiple of $6.$

3. $w\in\{w_{n-1}b,w_{n-2}bbaaaba,w_{n-3}bbaaabababbaa\}.$
\end{Lem}

\begin{proof}
For $n=3$ the lemma can be proved by a computer calculation. Suppose that we
have already proved the lemma for $n=k.$ Consider a word $w$ such that
$l(w)=6(k+1).$ If the conditions 1 or 2 does not hold for the word $w$ then
they fail for the word consisting of the first $6k$ letters of the word $w.$
Hence there exists a word $u\in S$ such that $l(u)=6$ and
$w\in\{w_{k-3}(bbaaba)^2bu,w_{k-3}(bbaaba)bbaaabau,w_{k-3}bbaaabababbaau\}.$
Then Lemmas 2 and 3 imply that the condition 3 is satisfied.
\end{proof}

\begin{Lem} Let $v\in S$, $l(v)=6n+r$, $0\le n$, $0\le r<6$ and $l(v)\not =11.$
Then $m(v)\le 2n+[3/2+r/4].$ If $l(v)=11$ then $m(v)\le 5.$
\end{Lem}
\begin{proof}
Remark that for $k=6n+r$ we get $(k+1)/3\le 2n+[3/2+r/4]\le (k+4)/3$ and if
$k=11$ then $5\le (k+4)/3.$ Also remark that $x\le 2n+[3/2+r/4]$ for each
positive integer $x<(k+4)/3$. For $l(v)\le 29$ the lemma is proved by the
computer calculation, see the above table. Suppose that we have already proved
the lemma for all words $v$ with $l(v)\le k$, where $k\ge 29.$

Let $l(v)=k+1=6n+r.$ Then $n\ge 5.$ Without loss of generality we may suppose that $v\in aS.$
We consider the following cases:

1. There exist words $v_1\in S,v_2\in S^1$ such that $v=v_1v_2$ and
$m(v_1)<l(v_1)/3.$ Then $m(v)\le m(v_1)+m(v_2)<l(v_1)/3+(l(v_2)+4)/3=
(l(v)+4)/3.$ Therefore $m(v)\le 2n+[3/2+r/4].$

2. There exist words $v_1,v_2\in S$ such that $v=v_1v_2,$ $m(v_1)\le
l(v_1)/3,l(v_2)\ge 12$ and $l(v_1)=6t.$ Then $m(v)\le m(v_1)+m(v_2)\le
2t+2(n-t)+[3/2+r/4]= 2n+[3/2+r/4].$

3. The cases 1 and 2 do not hold. Let $v=v_1v_2$ where
$l(v_1)=6(n-2),l(v_2)=12+r.$ Then Lemma 5 implies that
$v_1\in\{w_{n-3}b,w_{n-4}bbaaaba, w_{n-5}bbaaabababbaa\}.$ If there exist words
$v',w'\in S$ such that $v_1v_2=w_{n-5}v'w',l(v')<5$ and $m_{l(v')}+m(w')\le
[17/2+l(v_2)/4]$ then Lemma 2 implies that $m(v)\le m(w_{n-5}v')+m(w')\le
2(n-5)+m_{l(v')}+m(w')\le 2(n-5)+[17/2+3+r/4]=2n+[3/2+r/4].$ In the opposite
case Lemma 4 applied to the last 25+r letters of the word $w$ implies that
$v_1v_2\in\{w_{n-3}bbaaababbbaaab(abba),w_{n-2}bbaaabab(abba)\}.$ Then $m(v)\le
((l(v)-4)+4)/3+1=l(v)/3+1<(l(v)+4)/3$ and hence $m(v)\le 2n+[3/2+r/4].$
\end{proof}

The following lemmas will be uses to prove the lower bound.

\begin{Lem} For every $n\ge 0$ the word $(bbaaba)^n$ does not contain
a palindrome $p$ with $l(p)\ge 5.$
\end{Lem}
\begin{proof}
Put $v=bbaaba.$ If $v^n$ contains a palindrome $p$ with $l(p)\ge 2,$ then $v^n$
also contains a palindrome $p'$ such that $l(p')=l(p)-2.$ Therefore it suffices
to show that $v^n$ does not contain a palindrome $p$ with $l(p)\in\{5,6\}.$
Suppose the converse. Since the length of $p$ does not exceed the length of $v$
then we can find two consecutive subwords $v_1=v_2=v$ of $v^n$ such that $p$ is
a subword of $v_1v_2$. Thus $v^2$ also contains a palindrome $p$ such that
$l(p)\in\{5,6\}.$ The straight check shows the opposite.
\end{proof}

\begin{Lem} Let $n=6t+5+r$, $t\ge 1$, $0\le r<6.$ Suppose that the word $u_n$ consists of the first $n$
letters of the word $w_{t+1}.$ Then $m(u_n)=2t+m_r.$
\end{Lem}
\begin{proof}
Let $t\ge 1$ and $u_n=u_{n-k}p_k,$ where $p_k$ is a palindrome with $l(p_k)=k.$
Lemma 6 implies that $k\le 4.$ Therefore the following cases are possible:

If $n=6t+5$ then $u_n=w_{t-1}bbaaba.$ Hence $p_k=a$ or $p_k=aba$ and
$m(u_{6t+5})=\min\{m(u_{6t+4}),m(u_{6t+2})\}+1.$

If $n=6t+6$ then $u_n=w_{t-1}bbaabab.$ Hence $p_k=b$ or $p_k=bab$ and
$m(u_{6t+6})=\min\{m(u_{6t+5}),m(u_{6t+3})\}+1.$

If $n=6t+7$ then $u_n=w_{t-1}bbaababb.$ Hence $p_k=b$ or $p_k=bb$ and
$m(u_{6t+7})=\min\{m(u_{6t+6}),m(u_{6t+5})\}+1.$

If $n=6t+8$ then $u_n=w_{t-1}bbaababba.$ Hence $p_k=a$ or $p_k=abba$ and
$m(u_{6t+8})=\min\{m(u_{6t+7}),m(u_{6t+4})\}+1.$

If $n=6t+9$ then $u_n=w_{t-1}bbaababbaa.$ Hence $p_k=a$ or $p_k=aa$ and
$m(u_{6t+9})=\min\{m(u_{6t+8}),m(u_{6t+7})\}+1.$

If $n=6t+10$ then $u_n=w_{t-1}bbaababbaab.$ Hence $p_k=b$ or $p_k=baab$ and
$m(u_{6t+10})=\min\{m(u_{6t+9}),m(u_{6t+6})\}+1.$

The computer calculation shows that $m(u_8)=3$, $m(u_9)=4$, $m(u_{10})=4.$
Therefore $m(u_{11})=4$, $m(u_{12})=5$, $m(u_{13})=5$, $m(u_{14})=5$,
$m(u_{15})=6$, $m(u_{16})=6.$ This proves the lemma for $t=1.$

Suppose that the lemma is already proved for $t=k.$ Then for $t=k+1$ we obtain:

$m(u_{6k+11})=\min\{m(u_{6k+10}),m(u_{6k+8})\}+1=2k+4=2(k+1)+m_0.$

$m(u_{6k+12})=\min\{m(u_{6k+11}),m(u_{6k+9})\}+1=2k+5=2(k+1)+m_1.$

$m(u_{6k+13})=\min\{m(u_{6k+11}),m(u_{6k+12})\}+1=2k+5=2(k+1)+m_2.$

$m(u_{6k+14})=\min\{m(u_{6k+13}),m(u_{6k+10})\}+1=2k+5=2(k+1)+m_3.$

$m(u_{6k+15})=\min\{m(u_{6k+14}),m(u_{6k+13})\}+1=2k+6=2(k+1)+m_4.$

$m(u_{6k+16})=\min\{m(u_{6k+15}),m(u_{6k+12})\}+1=2k+6=2(k+1)+m_5.$
\end{proof}

\begin{Lem} For every number $n\ge 0$ we have
$m(aabab(bbaaba)^nbbaaababb)=2n+6.$
\end{Lem}
\begin{proof}
For $n\le 1$ the lemma is proved by the computer calculation. Let $n>1.$ Put
$v_n=aabab(bbaaba)^nbbaa(ababb).$ We claim that if $p$ is a palindrome such
that $v_n=v'pv''$ and $l(v'')<5$ then $l(p)\le 5.$ This can be proved by the
straight check taking into account that for a palindrome $p$ whose ``center of
symmetry'' lies in the subword $(bbaaba)^nbb$ of the word $v_n$ we can apply
Lemma 7 to conclude that $l(p)\le 4.$

Let $p_1\dots p_{m(v_n)}$ be a decomposition of the word $v_n,$ where
$p_1,\dots ,p_{m(v_n)}$ are palindromes. Take a number $k$ such that
$l(p_1\dots p_k)\le 6n+5$ and $l(p_1\dots p_{k+1})>6n+5.$ Put $v'=p_1\dots
p_{k+1},$ $v''=p_{k+2}\dots p_{m(v_n)}.$ Since $l(p_{k+1})\le 5$ then one of
the following cases holds:

1. $v''=baaababb.$ Then Lemma 8 implies that $m(v')=2n+m_1$ and the computer
calculation shows that $m(v'')=3.$ Therefore $m(v')+m(v'')=2n+m_1+3=2n+6.$

2. $v''=aaababb.$ Then $m(v')=2n+m_2,$ $m(v'')=3.$ Therefore
$m(v')+m(v'')=2n+m_2+3=2n+6.$

3. $v''=aababb.$ Then $m(v')=2n+m_3,$ $m(v'')=3.$ Therefore
$m(v')+m(v'')=2n+m_3+3=2n+6.$

4. $v''=ababb.$ Then $m(v')=2n+m_4,$ $m(v'')=2.$ Therefore
$m(v')+m(v'')=2n+m_4+2=2n+6.$

Hence $m(v_n)=m(v')+m(v'')=2n+6.$
\end{proof}

To finish the proof of Theorem 1 let us make the following remarks. Let
$t=6n+r$, $n\ge 0$, $0\le r<6$ and $t\not=11.$ It is easy to verify that
$[\frac t6]+[\frac{t+4}6]+1=2n+[\frac 32+\frac r4].$ Lemma 6 implies that
$K(t)\le 2n+[3/2+r/4].$ Lemma 8 implies that if $n\ge 2$ and $r\not=2$ then
$K(t)\ge 2n+[3/2+r/4].$ Lemma 9 yields $K(t)\ge 2n+[3/2+r/4]$ provided $n\ge 2$
and $r=2$. Finally, the computer calculation shows that $K(11)=5$ and
$K(t)=2n+[3/2+r/4]$ provided $n\le 1$ and $t\not=11$.

\section{Proof of the Theorem 2}

We shall use the following {\it Subadditive Lemma} \cite{1}, \S 2.5.

\begin{lemmaw} Let $\{x_n\}$ be a sequence such that $0\le x_{m+n}\le x_m+x_n$
for every positive integer $m,n$. Then $\lim\limits_{n\to\infty}
x_n/n=\inf\limits_{n\in\N} x_n/n.$
\end{lemmaw}

 To apply this lemma, observe that for positive integer $n,m$ we have
\[\ol K(m+n)=2^{-m-n}\sum\{m(w):l(w)=m+n\}=\]
\[2^{-m-n}\sum\{m(w_1w_2):l(w_1)=m,l(w_2)=n\}\le\]
\[2^{-m-n}\sum\{m(w_1)+m(w_2):l(w_1)=m,l(w_2)=n\}=\]
\[2^{-m}\sum\{m(w_1):l(w_1)=m\}+
2^{-n}\sum\{m(w_2):l(w_2)=n\}=\ol K(m)+\ol K(n).\]

Then the subadditive lemma yields the existence of the limit $\ol
K=\lim\limits_{n\to\infty}\ol K(n)/n$ and an upper bound $\ol K\le\frac{\ol
K(21)}{21}=\frac{372487}{7\cdot2^{18}}=0.2030\dots.$

Let $n\ge 9.$ Next, we prove the lower bound for $\ol K.$ Observe that $2^n\ol
K(n)=\sum\limits_{k=1}^{K(n)} kx_k,$ where $x_k=|\{w:l(w)=n, m(w)=k\}|.$ In
order to estimate the sum $\sum\limits_{k=1}^{K(n)} kx_k,$ we shall use the
following

\begin{Lem}\label{kSum} Let $l\ge p$ and $x_1,\dots,x_{l+1},a_1,\dots,a_{p+1}$ be
nonnegative real numbers, $\sum\limits_{k=1}^{l+1}x_k=
\sum\limits_{k=1}^{p+1}a_k$ and $x_k\le a_k$ for all $1\le k\le p$. Then
$\sum\limits_{k=1}^{l+1}kx_k\ge\sum\limits_{k=1}^{p+1}ka_k.$
\end{Lem}

\begin{proof} Indeed, $\sum\limits_{k=1}^{l+1}kx_k-\sum\limits_{k=1}^{p+1}ka_k=
\sum\limits_{k=1}^{l+1}kx_k-\sum\limits_{k=1}^{p}ka_k
-(p+1)\left(\sum\limits_{k=1}^{l+1}x_k-\sum\limits_{k=1}^{p}a_k\right)=
\sum\limits_{k=1}^{l+1}(k-p-1)x_k+\sum\limits_{k=1}^{p}
(p+1-k)a_k=\sum\limits_{k=p+1}^{l+1}(k-p-1)x_k+
\sum\limits_{k=1}^{p}(p+1-k)(a_k-x_k)\ge 0.$
\end{proof}

Now we are going to find numbers $a_k$ satisfying the conditions of
Lemma~\ref{kSum}. Let $w$ be a word such that $m(w)=k.$ Then $w=p_1\cdots p_k$
for some palindromes $p_1,\dots,p_k.$ For a fixed decomposition
$n=n_1+\dots+n_k$ as a sum of positive integers there exist
$\prod\limits_{i=1}^k 2^{\left[\frac{n_k+1}2\right]}\le 2^{\frac{n+k}2}$
different products of palindromes $p_1,\dots,p_k$ such that $l(p_i)=n_i$ for
every $i.$ Since there exist $n-1 \choose k-1$ decompositions of $n$ as a sum
of $k$ positive integer components then there exist not more than $a_k={n-1
\choose k-1} 2^{\frac{n+k}2}$ different products of $k$ palindromes. Hence
$x_k\le a_k.$

In fact the estimation $x_k\le a_k$ is too rough and there is a more subtle
estimation: if $w=p_1\dots p_k$ for some palindromes $p_1,\dots,p_k$ and
$k<n/2$ then there exists a palindrome $p_i$ such that $l(p_i)>2.$ Let
$p_i=xp'_ix,x\in\{a,b\}.$ Then there exists a decomposition $w=p_1\dots
p_{i-1}xp'_ixp_{i+1} \dots p_k$ of the word $w$ as a product of $k+2$
palindromes. If $k+2<n/2$ then there exists a decomposition of the word $w$ as
a product of $k+4$ palindromes and so forth. Since $K(n)<\frac{n}2$ for $n\ge
9$ we get $x_k\le x_k+x_{k-2}+x_{k-4}+\dots\le a_k$ for $n\ge 9$ and $k\le
K(n).$

There exists $p=p(n)$ such that $\sum\limits_{k=1}^p a_k\le
\sum\limits_{k=1}^{K(n)}x_k =2^n,$ $\sum\limits_{k=1}^{p+1}a_k> 2^n.$ For $1\le
k\le p$ put
$\delta_k=\frac{a_k}{a_{k+1}}=\frac{(n-k-1)!k!}{\sqrt{2}(k-1)!(n-k)!}=
\frac{k}{\sqrt 2(n-k)}.$ Since the sequence $\delta_k$ strictly increases then
for all $l\le k$ we have $a_l=a_{k+1}\delta_k\delta_{k-1}\cdots\delta_l\le
a_{k+1}\delta_k^{k+1-l}.$ Since $p\le K(n)<\frac n2$ for $n\ge 9$ then
$\delta_k<\frac 1{\sqrt 2}<1$ for every $k$. Therefore
$2^n<\sum\limits_{k=1}^{p+1} a_k\le a_{p+1}(1+\delta_p+\dots+
\delta_p^p)<\frac{a_{p+1}}{1-\delta_p}.$  Hence $a_{p+1}={{n-1} \choose p}
2^{\frac{n+p+1}2}>2^n(1-\delta_p).$
 Since $e^{\frac 1{12m}}\sqrt{2\pi m}\left (\frac me\right)^m>m!>
\sqrt{2\pi m}\left (\frac me\right)^m$ (see 21.4-2 in \cite{6}) for all
positive integer $m$ we obtain the estimation

\[e^{\frac 1{12(n-1)}}\sqrt{2\pi
 (n-1)}\left(\frac {n-1}e\right)^{n-1}2^{\frac{n+p+1}2}>a_{p+1}>2^n(1-\delta_p)>\]
\[\sqrt{2\pi p}\left(\frac pe\right)^p \sqrt{2\pi(n-1-p)}\left(\frac
{n-1-p}e\right)^{n-1-p}2^n (1-\delta_p),\] that implies

\[\frac 1{12(n-1)}+\frac 12\ln {2\pi(n-1)} +
(n-1)(\ln {(n-1)}-1)+\frac{p+1-n}2\ln 2>\]
\[\frac 12\ln{2\pi p}+p(\ln p-1)+\frac 12\ln {2\pi(n-1-p)}+
(n-1-p)(\ln(n-1-p)-1)+\]
\[\ln(1-\delta_p).\]
Let $\theta_n=\frac{p(n)}{n-1}.$ Put $r(n)=\frac 1{12(n-1)}+\frac 12\ln
{2\pi(n-1)}-\frac 12\ln{2\pi p}-\frac 12\ln {2\pi(n-1-p)}-\ln(1-\delta_p)$.
Then $r(n)=o(n-1)$ and
\[(n-1)(\ln {(n-1)}-1)+\frac{(\theta_n-1)(n-1)}2\ln 2+r(n)>\]
\[\theta_n(n-1)(\ln{\theta_n}+\ln{(n-1)} -1)+\]
\[(n-1)(1-\theta_n)
(\ln{(1-\theta_n)} +\ln{(n-1)}-1).\] This implies that
$f(\theta_n)>r(n)/(n-1),$ where $f(\theta)=\frac{\theta-1}2\ln
2-\theta\ln\theta-(1-\theta)\ln{(1-\theta)}$, $f(0)=\lim\limits_{\theta\to+0}
f(\theta)=-\frac{\ln 2}2.$

Let $\ol\theta=\ol{\lim\limits_{n\to\infty}}\theta_n.$ By the continuity of the
map $f$ on $[0;1),$ we get $f(\ol \theta)=
\ol{\lim\limits_{n\to\infty}}f(\theta_n)\ge
\lim\limits_{n\to\infty}r(n)/(n-1)=0.$ Since
$0\le\ol\theta\le\lim\limits_{n\to\infty}\frac{K(n)}n=\frac13$ and
$f'(\theta)=\frac{\ln 2}2-\ln\theta+\ln{(1-\theta)}>0$ for $0<\theta\le\frac
13$ we conclude that $\ol\theta\ge \theta',$ where $\theta'$ is the unique root
of the equation $f(\theta)=0$ on the segment $[0;\frac 13].$ Computer
calculation shows that $\theta'=0.09488\dots.$

Using the inequalities $\sum\limits_{k=1}^{K(n)} x_k\le 2^n$, $x_k\le a_k$ for
$k\le p,$ $a_{p-1}\le\frac{2^n}{1+1/\delta_{p-1}}$, Lemma 11 and the equality
$\delta_{p-1}=\frac{p-1}{\sqrt2(n-p+1)}=
\frac{\theta_n}{\sqrt2(1-\theta_n)}+o(1)$ we obtain

\[\sum\limits_{k=1}^{K(n)} kx_k\ge
\sum\limits_{k=1}^p ka_k+\left(2^n-\sum\limits_{k=1}^p a_k\right)(p+1)=
2^n+ \sum\limits_{k=1}^p (k-1)a_k+\]

\[\left(2^n-\sum\limits_{k=1}^p a_k\right)p\ge
\left(2^n-\sum\limits_{k=1}^{p-1} a_k\right)p\ge
\left(2^n-a_{p-1}\sum\limits_{k=1}^{p-1}\delta_{p-2}^{k-1}\right)p\ge\]

\[\left(2^n-\frac{a_{p-1}}{1-\delta_{p-2}}\right)p\ge
\left(2^n-\frac{\frac{2^n}{1+1/\delta_{p-1}}}{1-\delta_{p-2}}\right)p\ge
\left(2^n-\frac{\frac{2^n}{1+1/\delta_{p-1}}}{1-\delta_{p-1}}\right)p=\]

\[\left(1-\frac{\delta_{p-1}}{1-\delta_{p-1}^2}\right)2^np=
\left(1-\frac{\frac{\theta_n}{\sqrt 2(1-\theta_n)}}
{1-\frac{\theta_n^2}{2(1-\theta_n)^2}}+o(1)\right)2^np=
g(\theta_n)2^nn+o(2^nn),\]
where $g(\theta)=\theta-\frac{\sqrt 2\theta^2(1-\theta)}{\theta^2-4\theta+2}.$
Computer calculation shows that the function $g'(x)$ has two real
roots $x_1=0.313,x_2=5.83.$ Therefore $g(x)$ increases for $0\le x\le x_1$ and
decreases for $x_1\le x\le\frac 13.$ Since $\theta'\le\ol\theta\le\frac 13$ then
$g(\ol\theta)\ge\min{(g(\theta'),g(\frac 13))}=\min{(0.08781\dots,0.199)}=g(\theta').$

Let $\{n_l\}$ be a sequence such that $\theta_{n_l}\to\ol\theta.$ Then
$g(\theta_{n_l})\to g(\ol\theta)$ and therefore
\[\ol K=\lim\limits_{l\to\infty}
\frac{\sum\limits_{k=1}^{K(n_l)} kx_k}{2^{n_l}n_l}
\ge\lim\limits_{l\to\infty} g(\theta_{n_l})
=g(\ol\theta)=0.08781\dots.\]

\section*{Acknowledgements}
Author is very grateful to Taras Banakh for the help in preparation of the
manuscript and valuable remarks, to Oleg Verbitsky for valuable remarks and to
Aleksandr Spivak for the program verification.

\end{document}